\documentclass[11pt]{article}

\usepackage{amsmath,amsthm}
\usepackage{amssymb}
\usepackage{hyperref}
\usepackage{mathrsfs}
\usepackage{color} 
\usepackage{url}

\newtheorem{theorem}{Theorem}
\newtheorem{Prop}{Proposition}
\newtheorem{Cor}{Corollary}
\newtheorem{Lem}{Lemma}

\allowdisplaybreaks

\begin{document} 

\title{Explicit expressions for the related numbers of higher order Appell polynomials}  

\author{
Su Hu
\\
\small Department of Mathematics\\
\small South China University of Technology\\ 
\small Guangzhou, Guangdong 510640 China\\ 
\small \texttt{mahusu@scut.edu.cn}\\\\  
Takao Komatsu
\\ 
\small School of Mathematics and Statistics\\
\small Wuhan University\\
\small Wuhan 430072 China\\
\small \texttt{tkomatsu31@msn.com, komatsu@whu.edu.cn}
}

\date{
\small MR Subject Classifications: 11B68, 11B37, 11C20, 15A15, 33C15.}

\maketitle 

\begin{abstract}
In this note,  by using the  Hasse-Teichm\"uller derivatives,  we obtain  two explicit expressions for the related numbers of higher order Appell polynomials. 
One of them presents a determinant expression for the related numbers of higher order Appell polynomials, 
which involves several determinant expressions of special numbers, such as the higher order generalized hypergeometric Bernoulli and Cauchy numbers,
thus recovers the classical determinant expressions of Bernoulli and Cauchy numbers stated in an article by Glaisher in 1875.

\end{abstract}

\section{Introduction} 

The  Bernoulli numbers $B_{n}$ are defined by the generating function 
\begin{equation*}
\frac{t}{e^t-1}=\frac{1}{\frac{e^t-1}{t}}=\sum_{n=0}^\infty B_n\frac{t^n}{n!}\,
\end{equation*}
and the Bernoulli polynomials $B_{n}(z)$ are defined by 
\begin{equation}\label{BerP}
\frac{te^{zt}}{e^t-1}=\frac{e^{zt}}{\frac{e^t-1}{t}}=\sum_{n=0}^\infty B_n(z)\frac{t^n}{n!}.
\end{equation}
These numbers and polynomials have a long history, which arise from Bernoulli's calculations of power sums in 1713, that is, 
$$
\sum_{j=1}^{m}j^{n}=\frac{B_{n+1}(m+1)-B_{n+1}}{n+1}
$$ 
(see \cite[p.5, (2.2)]{Sun}). They have many applications in modern number theory, such as modular forms and Iwasawa theory \cite{Iwa}. 

For $r\in\mathbb{N}$, in 1924, N\"orlund \cite{Norlund} generalized (\ref{BerP}) to give a definition of  higher order Bernoulli polynomials and numbers 
\begin{equation*}
\Big(\frac{t}{e^t-1}\Big)^{r}e^{zt}=\frac{e^{zt}}{\big(\frac{e^t-1}{t}\big)^{r}}=\sum_{n=0}^\infty B_n^{(r)}(z)\frac{t^n}{n!}.
\end{equation*}  
We also have a similar expression of multiple power sums $$\sum_{l_{1},\ldots,l_{n}=0}^{m-1}(t+l_{1}+\cdots+\l_{n})^{k}$$ in terms of higher order Bernoulli polynomials (see \cite[Lemma 2.1]{HKI}).

The Euler polynomials are defined by the generating function  
\begin{equation*}
\frac{2e^{zt}}{e^t+1}=\frac{e^{zt}}{\frac{e^t+1}{2}}=\sum_{n=0}^\infty E_n(z)\frac{t^n}{n!}.
\end{equation*}
These polynomials were introduced by Euler who studied the alternating power sums, that is,
\begin{equation*}
\sum_{j=1}^{m}(-1)^{j+1}j^{n}=-\frac{(-1)^{{m}}E_{n}(m+1)+E_{n}(0)}{2}
\end{equation*} 
(see \cite[p.5, (2.3)]{Sun}). 
 We may also define the higher order Euler polynomial as follows 
\begin{equation*}
\Big(\frac{2}{e^t+1}\Big)^{r}e^{zt}=\frac{e^{zt}}{\big(\frac{e^t+1}{2}\big)^{r}}=\sum_{n=0}^{\infty} E_n^{(r)}(z)\frac{t^n}{n!}.
\end{equation*}
 
In 1880, Appell \cite{Appell} found that the Bernoulli and Euler polynomials are the sequences $(A_{n}(z))_{n\geq 0}$ in the polynomial ring $\mathbb{C}[z],$ which satisfy the recurrence relation $A_{n}^{'}(z)=nA_{n-1}(z)$ for $n\geq 1$ and $A_{0}(z)$ is a non-zero constant polynomial. And the sequence $(A_{n}(z))_{n\geq 0}$, now known as 
Appell polynomials, can also be defined equivalently by using generating functions, that is, let $S(t)=\sum_{n=0}^{\infty}a_{n}\frac{t^{n}}{n!}$ be a formal power series in $\mathbb{C}[[t]]$ and $a_{0}\neq 0$, we have 
\begin{equation}\label{Appell}
\sum_{n=0}^{\infty}A_{n}(z)\frac{t^{n}}{n!}=S(t)e^{zt}.
\end{equation} 
Recently,  Bencherif, Benzaghou and Zerroukhat \cite{BBZ} established an identity for some  Appell polynomials generalizing explicit formulas for generalized Bernoulli numbers and polynomials.  

Since $a_{0}\neq 0$, there exists the formal power series (for some $d_n\in\mathbb C$) 
\begin{equation}\label{ft} 
f(t)=\frac{1}{S(t)}=\sum_{n=0}^{\infty}d_{n}\frac{t^{n}}{n!}
\end{equation}
in $\mathbb{C}[[t]]$, and (\ref{Appell}) becomes  
\begin{equation}\label{Appell2}
\frac{e^{zt}}{f(t)}=\sum_{n=0}^{\infty}A_{n}(z)\frac{t^{n}}{n!}.
\end{equation}  
We can also define the higher order Appell polynomials by the generating function 
\begin{equation}\label{Appellh}
\frac{e^{zt}}{(f(t))^{r}}=\sum_{n=0}^{\infty}A_{n}^{(r)}(z)\frac{t^{n}}{n!}
\end{equation}  
(see \cite[Th\'eor\`eme 1.1]{BBZ}).
As in the classical case, we call $a_{n}^{(r)}=A_{n}^{(r)}(0)$ the related numbers of higher order Appell polynomials, that is,   
\begin{equation}\label{Appellhn}
\frac{1}{(f(t))^{r}}=\sum_{n=0}^{\infty}a_{n}^{(r)}\frac{t^{n}}{n!} 
\end{equation} 
and $a_{n}=a_{n}^{(1)}$ the related numbers of Appell polynomials. 

In what follows, we  assume $d_{0}=1$ to normalize the expansion of $f(t)$ in (\ref{ft}). 

The above definitions unify several generalizations of Bernoulli numbers and polynomials. For example, 
\begin{enumerate} 
\item {\it The higher order generalized hypergeometric Bernoulli numbers and polynomials.} 
  
Let $(x)^{(n)}=x(x+1)\dots(x+n-1)$ ($n\ge 1$) with $(x)^{(0)}=1$ be the rising factorial. 
For positive integers $M$, $N$ and $r$, put 
$$
f(t)={}_1 F_1(M;M+N;t)=\sum_{n=0}^\infty\frac{(M)^{(n)}}{(M+N)^{(n)}}\frac{t^n}{n!}, 
$$ 
the confluent hypergeometric function (see \cite{D2002}), in (\ref{Appell}),
we obtain the higher order generalized hypergeometric Bernoulli polynomials 
 $B_{M,N,n}^{(r)}(z),$ that is, 
\begin{equation*}  
\frac{e^{z t}}{\bigl({}_1F_1(M;M+N;t)\bigr)^r}=\sum_{n=0}^\infty B_{M,N,n}^{(r)}(z)\frac{t^n}{n!}\,.
\end{equation*} 
When $z=0$,  $B_{M,N,n}^{(r)}=B_{M,N,n}^{(r)}(0)$ are the higher order generalized hypergeometric Bernoulli numbers.  
When $M=1$, 
the higher order hypergeometric Bernoulli polynomials $B_{N,n}^{(r)}(z)=B_{1,N,n}^{(r)}(z)$ are studied by Hu and Kim in \cite{HK}.  
When $r=M=1$, we have $B_{N,n}(z)=B_{1,N,n}^{(1)}(z)$ and
\begin{equation}\label{hb} 
\frac{e^{z t}}{{}_1F_1(1;1+N;t)}=\frac{t^N e^{zt}/N!}{e^t-T_{N-1}(t)}=\sum_{n=0}^{\infty}B_{N,n}(z)\frac{t^{n}}{n!},
\end{equation}   
which is the hypergeometric Bernoulli polynomials defined by Howard in ~\cite{Ho1,Ho2} and was studied by Kamano in \cite{Kamano2}.
From (\ref{hb}), if $r=M=N=1$, we have $B_{n}(z)=B_{1,1,n}^{(1)}(z)$, which reduces to the classical Bernoulli polynomials (\ref{BerP}).

\item {\it The higher order generalized hypergeometric Cauchy numbers.}
For positive integers $M$, $N$ and $r$, put 
$$
f(t)={}_2 F_1(M,N;N+1;-t)=\sum_{n=0}^\infty\frac{(M)^{(n)}(N)^{(n)}}{(N+1)^{(n)}}\frac{(-t)^n}{n!}, 
$$ 
the Gauss hypergeometric function, in (\ref{Appellhn}), we obtain the higher order generalized hypergeometric Cauchy numbers $c_{M,N,n}^{(r)},$ that is, 
\begin{equation}
\frac{1}{\bigl({}_2 F_1(M,N;N+1;-t)\bigr)^r}=\sum_{n=0}^\infty c_{M,N,n}^{(r)}\frac{t^n}{n!}\,
\label{gen:ghgbp-r}
\end{equation}  
which has been studied by Komatsu and Yuan in \cite{KY}.  
When $r=1$, $c_{M,N,n}=c_{M,N,n}^{(1)}$ are the generalized hypergeometric Cauchy numbers which has been studied by Komatsu in \cite{Ko3}.  
When $r=M=N=1$, $c_{n}=c_{1,1,n}^{(1)}$ are classical Cauchy numbers  (see \cite[p. 383. Eq. (1)]{KY}),  
which can also be defined by the well-known generating function
\begin{equation}\label{Cauchy}
\frac{t}{\log(1+t)}=\sum_{n=0}^{\infty}c_{n}\frac{t^n}{n!}
\end{equation}
(see \cite{Zhao}).
\end{enumerate} 
 
In 1875, Glaisher  stated the following  classical determinant expression of  Bernoulli numbers $B_{n}$ in an article  (\cite[p.53]{Glaisher}):
\begin{equation}\label{Berd-g}
B_n=(-1)^n n!\left|
\begin{array}{ccccc} 
\frac{1}{2!}&1&&&\\  
\frac{1}{3!}&\frac{1}{2!}&&&\\ 
\vdots&\vdots&\ddots&1&\\ 
\frac{1}{n!}&\frac{1}{(n-1)!}&\cdots&\frac{1}{2!}&1\\ 
\frac{1}{(n+1)!}&\frac{1}{n!}&\cdots&\frac{1}{3!}&\frac{1}{2!}
\end{array} 
\right|\,. 
\end{equation}

In 2010, Costabile and Longo \cite[Theorem 2]{Constabile1} obtained a determinant expression of the first order Appell polynomials (\ref{Appell2}), and their result reduces to
a different   determinant expression for Bernoulli numbers:
\begin{equation*}
B_n=\frac{(-1)^n}{(n-1)!}\left|
\begin{array}{cccccc} 
\frac{1}{2}&\frac{1}{3}&\frac{1}{4}&\cdots&\frac{1}{n}&\frac{1}{n+1}\\  
1&1&1&\cdots&1&1\\ 
0&2&3&\cdots&n-1&n\\ 
0&2&\binom{3}{2}&\cdots&\binom{n-1}{2}&\binom{n}{2}\\ 
\vdots&\vdots&\vdots&\ddots&\vdots&\vdots\\ 
0&0&0&\cdots&\binom{n-1}{n-2}&\binom{n}{n-2}
\end{array} 
\right|\,
\end{equation*}
for $n\geq 1$ (also see \cite[p.7]{Constabile2}).

In 2016,  from the integral expression for the generating function of Bernoulli polynomials, Qi and Chapman \cite[Theorems 1 and 2]{Qi2016} got two closed forms for Bernoulli polynomials, which provided an explicit formula for computing these special polynomials in terms of Stirling numbers of the second kind $S(n,k)$. According to Wiki~\cite{Wiki} (also see \cite{Qi2016}),  ``In mathematics, a closed-form expression is a mathematical expression that can be evaluated in a finite number of operations. It may contain constants, variables, certain `well-known' operations (e.g., $+~-~\times ~\div$), and functions (e.g., $n$th root, exponent, logarithm, trigonometric functions, and inverse hyperbolic functions), but usually no limit."  It needs to mention that \cite[Theorem 2]{Qi2016} also  shows a new determinant expression of Bernoulli polynomials, which reduces to another
determinant expression of Bernoulli numbers. Recently, by directly applying the generating functions instead of their integral expressions, Hu and Kim  \cite{HKR2} obtained new closed form and determinant expressions of Apostol-Bernoulli polynomials~\cite{Apostol1, Apostol2, Luo}. 

In 2017, applying the  Hasse-Teichm\"uller derivatives~\cite{GN}, Komatsu and Yuan \cite{KY} presented a determinant expression of the higher order generalized hypergeometric Cauchy numbers defined as in (\ref{gen:ghgbp-r}) (see \cite[Theorem 4]{KY}). When $r=M=N=1$,
their formula recovers the classical determinant expression of Cauchy numbers $c_{n}$ (\ref{Cauchy}) which was also stated in the article of  Glaisher    (\cite[p.50]{Glaisher}).

In this note,  generalizing the methods in \cite{KY}, by using the  Hasse-Teichm\"uller derivatives~\cite{GN} (see Sec.2 below for a brief review),  we shall obtain  two explicit expressions for the related numbers of higher order Appell polynomials (Theorems \ref{prp1h} and~\ref{th1-r}). Theorem \ref{th1-r}  gives a determinant expression for the related numbers of higher order Appell polynomials defined as in (\ref{Appellhn}), which involves several determinant expressions of special numbers, such as the higher order generalized hypergeometric Bernoulli and Cauchy numbers stated above,
thus recovers the classical determinant expressions of Bernoulli and Cauchy numbers (see the last section).

\section{Hasse-Teichm\"uller derivatives}

We shall introduce the Hasse-Teichm\"uller derivatives 
in order to prove our  results.    

Let $\mathbb{F}$ be a field of any characteristic, $\mathbb{F}[[z]]$ the ring of formal power series in one variable $z$, and $\mathbb{F}((z))$ the field of Laurent series in $z$. Let $n$ be a nonnegative integer. We define the Hasse-Teichm\"uller derivative $H^{(n)}$ of order $n$ by 
\begin{equation*}
H^{(n)}\left(\sum_{m=R}^{\infty} c_m z^m\right)
=\sum_{m=R}^{\infty} c_m \binom{m}{n}z^{m-n}
\end{equation*}
for $\sum_{m=R}^{\infty} c_m z^m\in \mathbb{F}((z))$, 
where $R$ is an integer and $c_m\in\mathbb{F}$ for any $m\geq R$. Note that $\binom{m}{n}=0$ if $m<n$.  

The Hasse-Teichm\"uller derivatives satisfy the product rule \cite{Teich}, the quotient rule \cite{GN} and the chain rule \cite{Hasse}. 
One of the product rules can be described as follows.  
\begin{Lem}[{\cite{Teich, KY}}]
For $f_i\in\mathbb F[[z]]$ $(i=1,\dots,k)$ with $k\ge 2$ and for $n\ge 1$, we have 
$$
H^{(n)}(f_1\cdots f_k)=\sum_{i_1+\cdots+i_k=n\atop i_1,\dots,i_k\ge 0}H^{(i_1)}(f_1)\cdots H^{(i_k)}(f_k)\,. 
$$ 
\label{productrule2}
\end{Lem} 

The quotient rules can be described as follows.  
\begin{Lem}[\cite{GN, KY}]
For $f\in\mathbb F[[z]]\backslash \{0\}$ and $n\ge 1$,  
we have 
\begin{align} 
H^{(n)}\left(\frac{1}{f}\right)&=\sum_{k=1}^n\frac{(-1)^k}{f^{k+1}}\sum_{i_1+\cdots+i_k=n\atop i_1,\dots,i_k\ge 1}H^{(i_1)}(f)\cdots H^{(i_k)}(f)
\label{quotientrule1}
\\ 
&=\sum_{k=1}^n\binom{n+1}{k+1}\frac{(-1)^k}{f^{k+1}}\sum_{i_1+\cdots+i_k=n\atop i_1,\dots,i_k\ge 0}H^{(i_1)}(f)\cdots H^{(i_k)}(f)\,.
\label{quotientrule2} 
\end{align}   
\label{quotientrules}
\end{Lem}

\section{Results and their proofs}  

Recall (\ref{Appellh}) and (\ref{ft}) with $d_{0}=1$ for a normalization.  

We have the following proposition on the recurrence for the related numbers of higher order Appell polynomials. 

\begin{Prop} 
For $n\geq 1$, we have
$$ 
\sum_{m=0}^n\sum_{i_1+\cdots+i_r=n-m\atop i_1,\dots,i_r\ge 0}\frac{d_{i_{1}}\cdots d_{i_{r}}}{i_1!\cdots i_r!}\frac{a_{m}^{(r)}}{m!}=0\,. 
$$ 
\label{prp0h-r} 
\end{Prop}  

\begin{proof} From (\ref{Appellh}) and (\ref{ft}), we have 
\begin{align*}  
1&=\left(\sum_{l=0}^\infty d_{l}\frac{t^l}{l!}\right)^r\left(\sum_{m=0}^\infty a_{m}^{(r)}\frac{t^m}{m!}\right)\\
&=\left(\sum_{l=0}^\infty\sum_{i_1+\cdots+i_r=l\atop i_1,\dots,i_r\ge 0}\frac{l!}{i_1!\cdots i_r!}d_{i_{1}}\cdots d_{i_{r}}\frac{t^{l}}{l!}\right)\left(\sum_{m=0}^\infty a_{m}^{(r)}\frac{t^m}{m!}\right)\\
&=\sum_{n=0}^\infty\sum_{m=0}^n\binom{n}{m}\sum_{i_1+\cdots+i_r=n-m\atop i_1,\dots,i_r\ge 0}\frac{(n-m)!d_{i_{1}}\cdots d_{i_{r}}}{i_1!\cdots i_r!}a_{m}^{(r)}\frac{t^n}{n!}\,. 
\end{align*} 
Comparing the coefficients of $t^{n}$ in the above equality, we get our result.
\end{proof}

By using Proposition \ref{prp0h-r}, we have
\begin{Cor}\label{recurrence} 
\begin{equation}  
a_{n}^{(r)}=-n!\sum_{m=0}^{n-1}\sum_{i_1+\cdots+i_r=n-m\atop i_1,\dots,i_r\ge 0}\frac{d_{i_{1}}\cdots d_{i_{r}}}{i_1!\cdots i_r!}\frac{a_{m}^{(r)}}{m!}
\label{eq0h}
\end{equation}  
with $a_{0}^{(r)}=1.$
\end{Cor}

We have an explicit expression for the related numbers of higher order Appell polynomials $a_{n}^{(r)}$.

\begin{theorem} 
For $n\ge 1$, we have 
$$
a_{n}^{(r)}=n!\sum_{k=1}^n (-1)^{k}\sum_{e_1+\cdots+e_k=n\atop e_1,\dots,e_k\ge1}D_r(e_1)\cdots D_r(e_k)\,,
$$
where 
\begin{equation} 
D_r(e)
=\sum_{i_1+\cdots+i_r=e\atop i_1,\dots,i_r\ge 0}\frac{d_{i_{1}}\cdots d_{i_{r}}}{i_1!\cdots i_r!}.  
\label{mre} 
\end{equation}  
\label{prp1h}  
\end{theorem}  
\begin{proof} 
We make an application of the Hasse-Teichm\"uller derivatives which were introduced in Sec.2 above.

Put $h(t)=\bigl(f(t)\bigr)^r$, where 
$$
f(t)=\sum_{j=0}^\infty d_{j}\frac{t^j}{j!}\,. 
$$ 
Since 
\begin{align*} 
\left.H^{(i)}(f)\right|_{t=0}&=\left.\sum_{j=i}^\infty \frac{d_{j}}{j!}\binom{j}{i}t^{j-i}\right|_{t=0}\\
&=
\frac{d_{i}}{i!}
\end{align*}  
by the product rule of the Hasse-Teichm\"uller derivative in Lemma \ref{productrule2}, we get 
\begin{align*} 
\left.H^{(e)}(h)\right|_{t=0}&=\sum_{i_1+\cdots+i_r=e\atop i_1,\dots,i_r\ge 0}\left.H^{(i_1)}(f)\right|_{t=0}\cdots\left.H^{(i_r)}(f)\right|_{t=0}\\
&=\sum_{i_1+\cdots+i_r=e\atop i_1,\dots,i_r\ge 0}\frac{d_{i_{1}}\cdots d_{i_{r}}}{i_1!\cdots i_r!}:=D_r(e)\,. 
\end{align*} 
Hence, by the quotient rule of the Hasse-Teichm\"uller derivative in Lemma \ref{quotientrules} (\ref{quotientrule1}), we have 
\begin{align*} 
\frac{a_{n}^{(r)}}{n!}&=\sum_{k=1}^n\left.\frac{(-1)^k}{h^{k+1}}\right|_{t=0}\sum_{e_1+\cdots+e_k=n\atop e_1,\dots,e_k\ge 1}\left.H^{(e_1)}(h)\right|_{t=0}\cdots\left.H^{(e_k)}(h)\right|_{t=0}\\
&=\sum_{k=1}^n(-1)^{k}\sum_{e_1+\cdots+e_k=n\atop e_1,\dots,e_k\ge1}D_r(e_1)\cdots D_r(e_k)\,. 
\end{align*} 
\end{proof}

Now, we show a determinant expression of $a_{n}^{(r)}$.   

\begin{theorem}
For $n\ge 1$, we have 
$$
a_{n}^{(r)}=(-1)^n n!\left|
\begin{array}{ccccc} 
D_r(1)&1&&&\\  
D_r(2)&D_r(1)&&&\\ 
\vdots&\vdots&\ddots&1&\\ 
D_r(n-1)&D_r(n-2)&\cdots&D_r(1)&1\\ 
D_r(n)&D_r(n-1)&\cdots&D_r(2)&D_r(1) 
\end{array} 
\right|\,. 
$$ 
where $D_r(e)$ are given in $(\ref{mre})$. 
\label{th1-r}  
\end{theorem}  
\begin{proof} 
The proof is an application of the inductive method.

For simplicity, put $A_{n}^{(r)}=\frac{(-1)^n a_{n}^{(r)}}{n!}$. Then, we shall prove that for any $n\ge 1$ 
\begin{equation}  
A_{n}^{(r)}=\left|
\begin{array}{ccccc} 
D_r(1)&1&&&\\  
D_r(2)&D_r(1)&&&\\ 
\vdots&\vdots&\ddots&1&\\ 
D_r(n-1)&D_r(n-2)&\cdots&D_r(1)&1\\ 
D_r(n)&D_r(n-1)&\cdots&D_r(2)&D_r(1) 
\end{array} 
\right|\,.
\label{aNnr}
\end{equation}   
When $n=1$, (\ref{aNnr}) is valid, because by Theorem \ref{prp1h} 
$$
D_r(1)=d_{1}=A_{1}^{(r)}\,. 
$$ 
Assume that (\ref{aNnr}) is valid up to $n-1$. Notice that 
by (\ref{eq0h}), we have 
$$A_{n}^{(r)}=\sum_{l=1}^n(-1)^{l-1}A_{n-l}^{(r)}D_r(l)\,. $$
Thus, by expanding the first row of the right-hand side (\ref{aNnr}), it is equal to 
\begin{align*} 
&D_r(1)A_{n-1}^{(r)}-\left|
\begin{array}{ccccc} 
D_r(2)&1&&&\\  
D_r(3)&D_r(1)&&&\\ 
\vdots&\vdots&\ddots&1&\\ 
D_r(n-1)&D_r(n-3)&\cdots&D_r(1)&1\\ 
D_r(n)&D_r(n-2)&\cdots&D_r(2)&D_r(1) 
\end{array} 
\right|\\
&=D_r(1)A_{n-1}^{(r)}-D_r(2)A_{n-2}^{(r)}\\
&\qquad +\left|
\begin{array}{ccccc} 
D_r(3)&1&&&\\  
D_r(4)&D_r(1)&&&\\ 
\vdots&\vdots&\ddots&1&\\ 
D_r(n-1)&D_r(n-4)&\cdots&D_r(1)&1\\ 
D_r(n)&D_r(n-3)&\cdots&D_r(2)&D_r(1) 
\end{array} 
\right|\\
&=D_r(1)A_{n-1}^{(r)}-D_r(2)A_{n-2}^{(r)}+\cdots
+(-1)^{n-2}\left|
\begin{array}{cc}
D_r(n-1)&1\\
D_r(n)&D_r(1)
\end{array} 
\right|\\
&=\sum_{l=1}^n(-1)^{l-1}D_r(l)A_{n-l}^{(r)}=A_{n}^{(r)}\,.
\end{align*} 
Note that $A_{1}^{(r)}=D_r(1)$ and $A_{0}^{(r)}=1$. 
\end{proof}  

When $r=1$, we have the following determinant expression for the related numbers of  Appell polynomials.  

\begin{theorem}  
For $n\ge 1$, we have 
\begin{align*}
a_{n}
=(-1)^n n!\left|
\begin{array}{ccccc} 
d_{1}&1&&&\\  
\frac{d_{2}}{2!}&d_{1}&&&\\ 
\vdots&\vdots&\ddots&1&\\ 
\frac{d_{n-1}}{(n-1)!}&\frac{d_{n-2}}{(n-2)!}&\cdots&d_{1}&1\\ 
\frac{d_{n}}{n!}&\frac{d_{n-1}}{(n-1)!}&\cdots&\frac{d_{2}}{2!}&d_{1}
\end{array} 
\right|\,.
\end{align*}  
\label{det:an}
\end{theorem}  

\section{Applications of Theorem \ref{th1-r}}
In this section, as applications of Theorem \ref{th1-r},  
we show several determinant expressions of special numbers, including the higher order generalized hypergeometric Bernoulli and Cauchy numbers introduced in Sec.1,
thus recovers the classical determinant expressions of Bernoulli and Cauchy numbers.

\begin{enumerate}
\item  {\it The generalized hypergeometric Bernoulli numbers.}

For positive integers $M$, $N$, put $$f(t)={}_1 F_1(M;M+N;t)=\sum_{n=0}^\infty\frac{(M)^{(n)}}{(M+N)^{(n)}}\frac{t^n}{n!},$$ in Theorem \ref{det:an}, we obtain the determinant
expression of the generalized hypergeometric Bernoulli numbers:
\begin{align*}
&B_{M,N,n}
=(-1)^n n!\\
&\times\left|
\begin{array}{ccccc} 
\frac{(M)^{(1)}}{(M+N)^{(1)}}&1&&&\\  
\frac{(M)^{(2)}}{2!(M+N)^{(2)}}&\frac{(M)^{(1)}}{(M+N)^{(1)}}&&&\\ 
\vdots&\vdots&\ddots&1&\\ 
\frac{(M)^{(n-1)}}{(n-1)!(M+N)^{(n-1)}}&\frac{(M)^{(n-2)}}{(n-2)!(M+N)^{(n-2)}}&\cdots&\frac{(M)^{(1)}}{(M+N)^{(1)}}&1\\ 
\frac{(M)^{(n)}}{n!(M+N)^{(n)}}&\frac{(M)^{(n-1)}}{(n-1)!(M+N)^{(n-1)}}&\cdots&\frac{(M)^{(2)}}{2!(M+N)^{(2)}}&\frac{(M)^{(1)}}{(M+N)^{(1)}} 
\end{array} 
\right|\,.
\end{align*}  
\label{det:ghbn}

Letting $M=N=1$ in the above, we obtain Glaisher's determinant expression of Bernoulli numbers (\ref{Berd-g}). 

\item {\it The  generalized hypergeometric Cauchy numbers.}

For positive integers $M$, $N$, put $$f(t)={}_2 F_1(M,N;N+1;-t)=\sum_{n=0}^\infty\frac{(M)^{(n)}(N)^{(n)}}{(N+1)^{(n)}}\frac{(-t)^n}{n!},$$ 
 in Theorem \ref{det:an}, we obtain the determinant
expression of the generalized hypergeometric Cauchy numbers:
\begin{align*}
&c_{M,N,n}
= n!\left|
\begin{array}{ccccc} 
\frac{M\cdot N}{N+1}&1&&&\\  
\frac{(M)^{(2)}N}{2!(N+2)}&\frac{M\cdot N}{N+1}&&&\\ 
\vdots&\vdots&\ddots&1&\\ 
\frac{(M)^{(n-1)}N}{(n-1)!(N+n-1)}&\frac{(M)^{(n-2)}N}{(n-2)!(N+n-2)}&\cdots&\frac{M\cdot N}{N+1}&1\\ 
\frac{(M)^{(n)}N}{n!(N+n)}&\frac{(M)^{(n-1)}N}{(n-1)!(N+n-1)}&\cdots&\frac{(M)^{(2)}N}{2!(N+2)}&\frac{M\cdot N}{N+1}
\end{array} 
\right|\,
\end{align*}  
\label{det:ghcn}
(also see \cite[Theorem 3]{KY}).
Letting $M=N=1$ in the above, we recover the classical determinant expression of Cauchy numbers (\cite[p.50]{Glaisher}):
\begin{align*}
&c_{n}
= n!\left|
\begin{array}{ccccc} 
\frac{1}{2}&1&&&\\  
\frac{1}{3}&\frac{1}{2}&&&\\ 
\vdots&\vdots&\ddots&1&\\ 
\frac{1}{n}&\frac{1}{n-1}&\cdots&\frac{1}{2}&1\\ 
\frac{1}{n+1}&\frac{1}{n}&\cdots&\frac{1}{3}&\frac{1}{2}
\end{array} 
\right|\,.
\end{align*}  
\label{det:cn}
 \end{enumerate}


\begin{thebibliography}{99}

\bibitem{Appell}
P. Appell, {\em 
Sur une classe de polyn\^{o}mes}, 
Ann. Sci. Ecole Norm. Sup. (2) {\bf 9} (1880), 119--144.

\bibitem{Apostol1} T.M. Apostol, On the Lerch zeta function, Pac. J. Math. \textbf{1} (1951), 161--167.  

\bibitem{Apostol2} T.M. Apostol, Addendum to “On the Lerch zeta function”. Pac. J. Math. \textbf{2} (1952), 10.

\bibitem{BBZ}  
F. Bencherif, B. Benzaghou and S. Zerroukhat, {\em 
Une identit\'e pour des polyn\^{o}mes d'Appell}, 
C. R. Math. Acad. Sci. Paris {\bf 355} (2017), 1201--1204. 


\bibitem{Constabile1} 
F. Costabile and E. Longo,  {\em 
A determinantal approach to Appell polynomials}, 
J. Comput. Appl. Math. {\bf 234} (2010), 1528--1542. 

\bibitem{Constabile2} 
F. Costabile, F. Dell'Accio and M.I. Gualtieri, {\em 
A new approach to Bernoulli polynomials,} 
Rend. Mat. Appl. {\bf 26} (2006), 1--12. 

 \bibitem{D2002} 
K. Dilcher, {\em 
Bernoulli numbers and confluent hypergeometric functions},  
Number Theory for the Millennium, I (Urbana, IL, 2000), 343--363, 
A K Peters, Natick, MA, 2002.  

\bibitem{Glaisher}  
J. W. L. Glaisher, {\em 
Expressions for Laplace's coefficients, Bernoullian and Eulerian numbers etc. as determinants}, 
Messenger (2) {\bf 6} (1875), 49-63. 
https://www.deutsche-digitale-bibliothek.de/item/S4VML72SWQQZ3U5CNG66KNUXZG3RW6VI?lang=de

\bibitem{GN} 
R. Gottfert and H. Niederreiter, {\em 
Hasse-Teichm\"uller derivatives and products of linear recurring sequences}, 
Finite Fields: Theory, Applications, and Algorithms (Las Vegas, NV, 1993), 
Contemporary Mathematics, vol. 168, American Mathematical Society, Providence, RI, 1994, pp.117--125. 

 

 \bibitem{Hasse} 
H. Hasse, {\em 
Theorie der h\"oheren Differentiale in einem algebraischen Funktionenk\"orper mit Vollkommenem Konstantenk\"orper bei beliebiger Charakteristik}, 
J. Reine Angew. Math. {\bf 175} (1936), 50--54.

 \bibitem{Ho1} 
F. T. Howard, {\em 
A sequence of numbers related to the exponential function},  
Duke Math. J. {\bf 34} (1967), 599--615.  

\bibitem{Ho2} 
F. T. Howard, {\em 
Some sequences of rational numbers related to the exponential function},  
Duke Math. J. {\bf 34} (1967), 701--716.  

\bibitem{HK} 
S. Hu and M.-S. Kim, {\em 
On hypergeometric Bernoulli numbers and polynomials}, Acta Math. Hungar. {\bf 154} (2018), 134--146.

\bibitem{HKR2} 
S. Hu and M.-S. Kim, {\em Two closed forms for the Apostol--Bernoulli polynomials}, Ramanujan J.  {\bf 46} (2018), 103--117.
\bibitem{Iwa} 
K. Iwasawa, {\em 
Lectures on $p$-Adic $L$-Functions}, 
Ann. of Math. Stud., vol. 74, Princeton Univ. Press, Princeton, 1972.

 
 \bibitem{Kamano2}  
K. Kamano, {\em  
Sums of products of hypergeometric Bernoulli numbers},  
J. Number Theory {\bf 130} (2010), 2259--2271.  

\bibitem{Ko3}  
T. Komatsu, {\em 
Hypergeometric Cauchy numbers},  
Int. J. Number Theory {\bf 9} (2013), 545--560. 

\bibitem{KY}  
T. Komatsu and P. Yuan, {\em 
Hypergeometric Cauchy numbers and polynomials},  
Acta Math. Hungar.  {\bf 153} (2017), 382--400.  

\bibitem{HKI} M.-S. Kim and S. Hu,
\textit{A $p$-adic view of multiple sums of powers}, Int. J. Number Theory \textbf{7} (2011), 2273--2288.

\bibitem{Luo} Q.-M. Luo, \textit{On the Apostol--Bernoulli polynomials}, Cent. Eur. J. Math. \textbf{2} (2004), 509--515.

\bibitem{Norlund} 
M. E. N\"orlund, {\em 
Differenzenrechnung}, 
Springer--Verlag, Berlin, 1924.

\bibitem{Qi2016} 
F. Qi and R. J. Chapman, {\em Two closed forms for the Bernoulli polynomials},  J. Number Theory {\bf 159} (2016), 89--100. 

\bibitem{Sun} 
Z.-W. Sun,
{\em Introduction to Bernoulli and Euler polynomials}, a lecture given in Taiwan on June 6, 2002, http://maths.nju.edu.cn/~zwsun/BerE.pdf.
 
 \bibitem{Teich}  
O. Teichm\"uller, {\em 
Differentialrechung bei Charakteristik $p$}, 
J. Reine Angew. Math. {\bf 175} (1936), 89--99. 

\bibitem{Wiki} Closed-form expression, \url{https://en.wikipedia.org/wiki/Closed-form\_expression}

 \bibitem{Zhao}  
F. Zhao, {\em 
Sums of products of Cauchy numbers}, 
Discrete Math. {\bf 309} (2009), 3830--3842.

 \end{thebibliography}
\end{document}